\documentclass[12pt]{article}
\usepackage{amsthm}
\usepackage{amssymb}
\usepackage{amsmath}

\newtheorem{thm}{Theorem}
\newtheorem{prop}{Proposition}
\newtheorem{defi}{Definition}
\newtheorem{cor}{Corollary}
\newtheorem{rem}{Remark}

\begin{document}
\title{An asymptotic formula of the divergent bilateral basic hypergeometric series }
\author{Takeshi MORITA\thanks{Graduate School of Information Science and Technology, Osaka University, 
1-1  Machikaneyama-machi, Toyonaka, 560-0043, Japan.} }
\date{}
\maketitle
\begin{abstract}
We show an asymptotic formula of the divergent bilateral basic hypergeometric series
 ${}_1\psi_0 (a;-;q,\cdot )$ with using the $q$-Borel-Laplace method. We also give the limit 
$q\to 1-0$ of our asymptotic formula.

\end{abstract}
\section{Introduction}
In this paper, we show an asymptotic formula of the divergent bilateral basic hypergeometric series 
\[{}_1\psi_0 (a;-;q;x ):=\sum_{n\in\mathbb{Z}} (a;q)_n\left\{(-1)^nq^{\frac{n(n-1)}{2}}\right\}^{-1}x^n.
\]
Here, $(a;q)_n$ is the $q$-shifted factorial;
\[(a;q)_n:=
\begin{cases}
1, &n=0, \\
(1-a)(1-aq)\dots (1-aq^{n-1}), &n\ge 1,\\
[(1-aq^{-1})(1-aq^{-2})\dots (1-aq^n)]^{-1}, &n\le -1
\end{cases}
\]
moreover, $(a;q)_\infty :=\lim_{n\to \infty}(a;q)_n$ and 
\[(a_1,a_2,\dots ,a_m;q)_\infty:=(a_1;q)_\infty (a_2;q)_\infty \dots (a_m;q)_\infty.\]

The $q$-shifted factorial $(a;q)_n$ is a $q$-analogue of the shifted factorial 
\[(\alpha )_n=\alpha \{\alpha +1\}\cdots \{\alpha +(n-1)\}.\]
The series ${}_1\psi_0(a;-;q,\cdot )$ is related to the bilateral basic hypergeometric series ${}_1\psi_1(a;b;q,\cdot )$. At first, we review the series
\[
 {}_1\psi_1(a;b;q,z):=\sum_{n\in\mathbb{Z}}\frac{(a;q)_n}{(b;q)_n}z^n.
\]
This series has Ramanujan's summation formula \cite{GR} 
\begin{equation}
 {}_1\psi_1(a;b;q,z)=\frac{(q,b/a,az,q/az;q)_\infty}{(b,q/a,z,b/az;q)_\infty}, \qquad 
\left|b/a\right|<|z|<1,\label{rsum}
\end{equation}
which was first given by Ramanujan \cite{Hardy}. This formula is considered as a extension of the $q$-binomial theorem \cite{GR,Askey};
\[\sum_{n\ge 0}\frac{(a;q)_n}{(q;q)_n}z^n=\frac{(az;q)_\infty}{(z;q)_\infty},\qquad |z|<1.\]
We also regard the summation \eqref{rsum} as a $q$-analogue of the bilateral binomial theorem \cite{Horn} discovered by M.~E.~Horn. If $\alpha$ and $\beta$ are complex numbers $\Re (\beta -\alpha )>1$ and $z$ is a complex number with $|z|=1$, then
\begin{equation}
\sum_{n\in\mathbb{Z}}\frac{(\alpha )_n}{(\beta )_n}z^n
=\frac{(1-z)^{\beta -\alpha -1}}{(-z)^{\beta -1}}\frac{\Gamma (1-\alpha )\Gamma (\beta )}{\Gamma (\beta -\alpha )}.\label{bbt1}
\end{equation}
We remark that the limit $q\to 1-0$ of \eqref{rsum} with suitable condition gives the bilateral binomial theorem \eqref{bbt1}. 

The series ${}_1\psi_1(a;b;q,z)$ satisfies the following $q$-difference equation 
\[
\left(\frac{b}{q}-az\right)u(qz)+(z-1)u(z)=0.
\]
In this paper, we study the degeneration of this equation as follows;
\[
\left(\frac{1}{q}-ax\right)\tilde{u}(qx)+x \tilde{u}(x)=0.
\]
This equation has a formal solution  
\begin{equation}
\tilde{u}(x)={}_1\psi_0(a;-;q,x).
\label{b2s}
\end{equation}
But $\tilde{u}(x)$ is a divergent around the origin \cite{GR} and its properties are not clear. In section three, we show an asymptotic formula of the series \eqref{b2s} with using the $q$-Borel-Laplace transformations. The $q$-Borel-Laplace transformations are introduced in the study of connection problems on linear $q$-difference equations between the origin and the infinity. 

Connection problems on $q$-difference equations with regular singular points are studied by G.~D.~Birkhoff \cite{Birkhoff}. The first example of the connection formula is given by G.~N.~Watson \cite{W} in 1910 as follows;
\begin{align}
{}_2 \varphi_1(a,b;c;q,&x )= 
\frac{(b,c/a;q)_\infty (a x,q/ a x;q)_\infty }{(c, b/a;q)_\infty (  x,q/   x;q)_\infty } 
{}_2 \varphi_1\left(a,aq/c;aq/b;q,cq/abx \right) \notag \\
&+\frac{(a,c/b;q)_\infty (b x, q/ b x;q)_\infty }{(c, a/b;q)_\infty (  x,q/   x;q)_\infty } 
{}_2 \varphi_1\left(b, bq/c; bq/a; q, cq/abx \right). \label{wat}
\end{align}
The function ${}_2\varphi_1(a,b;c;q,x)$ is the basic hypergeometric series 
\[{}_2\varphi_1(a,b;c;q,x)=\sum_{n\ge 0} \frac{(a,b;q)_n}{(c;q)_n(q;q)_n}x^n,\]
which is introduced by E.~Heine \cite{hein}. This series satisfies the second order linear $q$-difference equation 
\begin{equation}
(c-abqx)u(q^2x)-\left\{c+q-(a+b)qx\right\}u(qx)+q(1-x)u(x)=0.
\label{heine}
\end{equation}
around the origin. On the other hand, the equation \eqref{heine} has a fundamental system of solutions 
\begin{equation*}
v_1(x)=\frac{\theta (-ax)}{\theta (-x)}{}_2\varphi_1\left(a,\frac{aq}{c};\frac{aq}{b};q,\frac{cq}{abx}\right), 
\quad v_2(x)=\frac{\theta (-bx)}{\theta (-x)}{}_2\varphi_1\left(b,\frac{bq}{c};\frac{bq}{a};q,\frac{cq}{abx}\right)
\end{equation*}
around the infinity \cite{Z0}. Here, $\theta (\cdot )$ is the theta function of Jacobi (see section two).  In the connection formula \eqref{wat}, we remark that the connection coefficients are given by $q$-elliptic functions. 

Other connection formulae, especially connection formulae of irregular singular type $q$-special functions has not known for a long time. Recently, C.~Zhang gives some connection formulae by two different types of $q$-Borel transformations and $q$-Laplace transformations \cite{Z0, Z1, Z2}, which are introduced by J.~Sauloy\cite{sauloy}. We assume that $f(x)$ is a formal power series $f(x)=\sum_{n\ge 0}a_nx^n$, $a_0=1$.

\begin{defi}
For any power series $f(x)$, the $q$-Borel transformation of the first kind $\mathcal{B}_q^+$ is
\[\left(\mathcal{B}_q^+f\right)(\xi ):=\sum_{n\ge 0}a_nq^{\frac{n(n-1)}{2}}\xi^n=:\varphi (\xi ).\]
For an entire function $\varphi$, the $q$-Laplace transformation of the first kind $\mathcal{L}_{q,\lambda }^+$ is
\[\left(\mathcal{L}_{q,\lambda}^+\varphi\right)(x):=
\sum_{n\in\mathbb{Z}}\frac{\varphi (\lambda q^n)}{\theta \left(\frac{\lambda q^n}{x}\right)},\quad \lambda\not\in q^{\mathbb{Z}}.\]
Similarly, the $q$-Borel transformation of the second kind $\mathcal{B}_q^-$ is 
\[\left(\mathcal{B}_q^-f\right)(\xi ):=\sum_{n\ge 0}a_nq^{-\frac{n(n-1)}{2}}\xi^n=:g(\xi )\]
and the $q$-Laplace transformation of the second kind $\mathcal{L}_q^-$ is 
\[\left(\mathcal{L}_q^-g\right)(x):=\frac{1}{2\pi i}\int_{|\xi |=r}g(\xi )\theta\left(\frac{x}{\xi}\right)\frac{d\xi}{\xi}.\]
\end{defi}
We remark that each $q$-Borel transformation is a formal inverse of each $q$-Laplace transformation;
\[\mathcal{L}_{q,\lambda}^+\circ\mathcal{B}_q^+f=f,\qquad 
\mathcal{L}_q^-\circ\mathcal{B}_q^-f=f.\]

We can find applications of the $q$-Borel-Laplace method of the first kind in \cite{Z0,Z1,M4}. This summation method is powerful to deal with divergent type basic hypergeometric series and to study the $q$-Stokes phenomenon. Applications of the method of the second kind can be found in \cite{Z2,M0,M1,M4}.  But other examples, for example, applications to bilateral basic hypergeometric series, has not known. 

In this paper, we study an asymptotic formula of the divergent bilateral basic hypergeometric series ${}_1\psi_0(a;-;q,\cdot )$ from viewpoint of connection problems on linear $q$-difference equations. We show the following theorem.

\noindent
\textbf{Theorem.}\textit{ For any} $x\in\mathbb{C}^*\setminus [-\lambda ;q]$ \textit{, we have}
\[\hat{{}_1\psi_0}(a;-;\lambda ;q,x)=\frac{(q;q)_\infty}{\left(q/a;q\right)_\infty}\frac{\theta (aq\lambda )}{\theta (q\lambda )}\frac{\theta\left(ax/\lambda \right)}{\theta\left(x/\lambda\right)}\frac{1}{\left(1/ax;q\right)_\infty},\]
\textit{  where $1<|ax|$.}

Here, the function $\hat{{}_1\psi_0}(a;-;\lambda ;q,x)$ is the $q$-Borel-Laplace transform of the series ${}_1\psi_0(a;-;q,x)$ and the set $[-\lambda ;q]$ is the $q$-spiral which is given by $[\lambda ;q]:=\{\lambda q^k: k\in\mathbb{Z}\}$. We also show the limit $q\to 1-0$ of our asymptotic formula in section four. If we take a suitable limit $q\to 1-0$ of the theorem above, we formally obtain the asymptotic formula of the bilateral hypergeometric series ${}_1H_0$.

\section{Basic notations}
In this section, we fix our notations. We assume that $0<|q|<1$ and $\sigma_q$ is the $q$-difference operator such that $\sigma_qf(x)=f(qx)$. The basic hypergeometric series is
\begin{align}
{}_r\varphi_s(a_1,\dots ,a_r;&b_1,\dots ,b_r;q,x)\notag\\
&:=\sum_{n\ge 0}\frac{(a_1,\dots ,a_r;q)_n}{(b_1,\dots ,b_s;q)_n(q;q)_n}
\left\{(-1)^nq^{\frac{n(n-1)}{2}}\right\}^{1+s-r}x^n.
\label{bh}
\end{align}
The radius of convergence $\rho$ of the basic hypergeometric series is given by \cite{Koo}
\[\rho =
\begin{cases}
\infty ,& \textrm{if }\quad r<s+1, \\
1, & \textrm{if }\quad r=s+1, \\
0, & \textrm{if }\quad r>s+1.
\end{cases}
\]
The bilateral basic hypergeometric series is 
\begin{align}
{}_r\psi_s(a_1,\dots ,a_r;&b_1,\dots ,b_s;q,x)\notag\\
&:=\sum_{n\in\mathbb{Z}}\frac{(a_1,\dots ,a_r;q)_n}{(b_1,\dots ,b_s;q)_n}
\left\{(-1)^nq^{\frac{n(n-1)}{2}}\right\}^{s-r}x^n.
\label{bbh}
\end{align}
If $s<r$, the series \eqref{bbh} diverges for $x\not=0$ and if $r=s$, the series \eqref{bbh} converges $|b_1\dots b_s/a_1\dots a_r|<|x|<1$ (see \cite{GR} for more detail).
The series \eqref{bh} is a $q$-analogue of the generalized hypergeometric function 
\begin{align*}
{}_rF_s (\alpha_1,\dots ,\alpha_r&;\beta_1,\dots ,\beta_s;x)
:=\sum_{n\ge 0}\frac{(\alpha_1,\dots ,\alpha_r)_n}{(\beta_1,\dots ,\beta_s)_nn!}x^n
\end{align*}
and the series \eqref{bbh} is a $q$-analogue of the bilateral hypergeometric function
\[{}_rH_s(\alpha_1,\dots ,\alpha_r;\beta_1,\dots ,\beta_s;x)
:=\sum_{n\in\mathbb{Z}}\frac{(\alpha_1,\dots ,\alpha_r)_n}{(\beta_1;\dots ,\beta_s)_n}x^n.\]
By D'Alembert's ratio test, it can be checked that ${}_rH_r$ converges only for $|x|=1$ \cite{Slater}, provided that $\Re (\beta_1+\dots +\beta_r-\alpha_1-\dots -\alpha_r)>1$.

The $q$-exponential function $e_q(x)$ is
\[e_q(x):={}_1\varphi_0(0;-;q,x)
=\sum_{n\ge 0}\frac{x^n}{(q;q)_n}\]
The function $e_q(x)$ has the infinite product representation as follows;
\[e_q(x)=\frac{1}{(x;q)_\infty},\qquad |x|<1.\]
The limit $q\to 1-0$ of $e_q(x)$ is the exponential function \cite{GR};
\begin{equation}
\lim_{q\to 1-0}e_q(x(1-q))=e^x.
\label{limexp}
\end{equation}

The $q$-gamma function $\Gamma_q(x)$ is
\[\Gamma_q(x):=\frac{(q;q)_\infty}{(q^x;q)_\infty}(1-q)^{1-x},\qquad 0<q<1.\]
The limit $q\to 1-0$ of $\Gamma_q(x)$ gives the gamma gunction \cite{GR};
\begin{equation}
\lim_{q\to 1-0}\Gamma_q(x)=\Gamma (x).\label{limgamma}
\end{equation}
The theta function of Jacobi is given by
\[\theta (x):=\sum_{n\in\mathbb{Z}}q^{\frac{n(n-1)}{2}}x^n,\qquad \forall x\in\mathbb{C}^*.\]
Jacobi's triple product identity is 
\[\theta (x)=(q,-x,-q/x;q)_\infty\]
and the theta function satisfies the following $q$-difference equation
\[\theta (q^kx)=x^{-k}q^{-\frac{k(k-1)}{2}}\theta (x),\qquad \forall k\in\mathbb{Z}.\]
The inversion formula is
\begin{equation}
\theta (x)=x\theta (1/x).
\label{inv}
\end{equation}
We remark that $\theta (\lambda q^k/x)=0$ if and only if $x\in[-\lambda ;q]$. 
In our study, the following proposition \cite{Z1} is useful to consider the limit $q\to 1-0$ of our formula.
\begin{prop} For any $x\in\mathbb{C}^* (-\pi <\arg x<\pi )$, we have
\begin{equation}
\lim_{q\to 1-0}\frac{\theta (q^\alpha x)}{\theta (q^\beta x)}=x^{\beta -\alpha}\label{limt1}
\end{equation}
and 
\begin{equation}
\lim_{q\to 1-0}\frac{\theta \left(\dfrac{q^\alpha x}{(1-q)}\right)}{\theta \left(\dfrac{q^\beta x}{(1-q)}\right)}(1-q)^{\beta -\alpha}=
\left(\frac{1}{x}\right)^{\alpha -\beta}.\label{limt2}
\end{equation}
\end{prop}
We also use the limit \cite{GR} as follows;
\begin{equation}
\lim_{q\to 1-0}\frac{(xq^\alpha ;q)_\infty}{(x;q)_\infty}=(1-x)^{-\alpha},\qquad |x|<1.
\label{limbin}
\end{equation}

\section{An asymptotic formula of the divergent series ${}_1\psi_0(a;-;q,x)$}
In this section, we show an asymptotic formula of the divergent series ${}_1\psi_0(a;-;q,x)$. At first, we review Ramanujan's sum for ${}_1\psi_1(a;b;q,z)$ and its property. Ramanujan gives the following sum \cite{Hardy}
\begin{equation}
{}_1\psi_1(a;b;q,z)=\frac{(q,b/a,az,q/az;q)_\infty}{(b,q/a,z,b/az;q)_\infty},\qquad \left|b/a\right|<|z|<1.
\label{Rsum}
\end{equation}
We can regard this sum as a $q$-analogue of the bilateral binomial theorem \cite{Horn} as follows;

\noindent
\textbf{Theorem.} \textit{(Horn \cite{Horn})} \textit{If $z$ is a complex number such that $|z|=1$, we have}\begin{equation}
{}_1H_1(\alpha ;\beta ;z)=\frac{\Gamma (\beta )\Gamma (1-\alpha )}{\Gamma (\beta -\alpha )}\frac{(1-z)^{\beta -\alpha -1}}{(-z)^{\beta -1}},
\label{bbt}
\end{equation}
\textit{provided that $\Re (\beta -\alpha )>1$.}

Let $z$ is a complex number such that $-\pi <\arg z <\pi$. The summation \eqref{Rsum} is rewritten by
\begin{equation}
\frac{(q,b/a,az,q/az;q)_\infty}{(b,q/a,z,b/az;q)_\infty}
=\frac{(q,b/a;q)_\infty}{(b,q/a;q)_\infty}\frac{\theta (-az)}{\theta (-aqz/b)}\frac{(aqz/b;q)_\infty}{(z;q)_\infty}.
\label{bin3}
\end{equation}
In \eqref{bin3}, we put $a=q^\alpha$ and $b=q^\beta$ (with $\Re (\beta -\alpha )>1$), we obtain
\begin{equation}
{}_1\psi_1 (q^\alpha ;q^\beta ;q,z)
=\frac{\Gamma_q(\beta )\Gamma_q (1-\alpha )}{\Gamma_q(\beta -\alpha )}
\frac{\theta (q^\alpha (-z))}{\theta (q^{\alpha +1-\beta }(-z))}
\frac{(q^{\alpha +1-\beta}z;q)_\infty}{(z;q)_\infty},
\label{qbb}
\end{equation}
where $|q^{\beta -\alpha}|<|z|<1$. Combining \eqref{limgamma}, \eqref{limt1} and \eqref{limbin}, we can check out that the limit $q\to 1-0$ of \eqref{qbb} gives the bilateral binomial theorem \eqref{bbt}, provided that $-\pi <\arg z <\pi$.

The bilateral basic hypergeometric series ${}_1\psi_1(a;b;q,z)$ satisfies the $q$-difference equation
\begin{equation}
\left(\frac{b}{q}-az\right)u(qz)+(z-1)u(z)=0.
\label{bqua}
\end{equation}
We consider the degeneration of the equation \eqref{bqua} in the next section.
\subsection{Local solutions of the degenerated equation}
In the equation \eqref{bqua}, we put $z=bx$ and take the limit $b\to\infty$, we obtain the equation
\begin{equation}
\left\{\left(\frac{1}{q}-ax\right)\sigma_q+x\right\} \tilde{u}(x)=0.\label{bqua2}
\end{equation}
The formal solution of \eqref{bqua2} is
\begin{equation}\label{b2so}
\tilde{u}(x)={}_1\psi_0(a;-;q,x),
\end{equation}
which is divergent around the origin. We consider ``the basic hypergeometric type'' solution around the infinity. We assume that the solution around the infinity is given by
\[\tilde{u}_\infty (x):=\frac{\theta (ax)}{\theta (x)}v_\infty (x)=\frac{\theta (ax)}{\theta (x)}\sum_{n\ge 0}v_nx^{-n}, \qquad v_0=1.\]
By the $q$-difference equation of the theta function, we obtain the equation
\[\left\{\left(\frac{1}{aq}-x\right)\sigma_q+x\right\}v_\infty (x)=0.\]
We remark that the function $\theta (ax)/\theta (x)$ satisfies the following $q$-difference equation 
\[u(qx)=q^{-\alpha }u(x)\]
which is also satisfied by the function $u(x)=x^{-\alpha}$ with $\log_qa=\alpha$.
We can check out that the series $v_\infty (x)$ is
\[v_\infty (x)=e_q\left(\frac{1}{ax}\right)=\frac{1}{\left(1/ax;q\right)_\infty},\qquad \left|\frac{1}{ax}\right|<1.\]
Therefore, one of the solution of \eqref{bqua2} around the infinity is
\begin{equation}\label{b2si}
\tilde{u}_\infty(x)=\frac{\theta (ax)}{\theta (x)}e_q\left(\frac{1}{ax}\right).
\end{equation}
In the following section, we study the relation between these solutions \eqref{b2so} and \eqref{b2si} with using the $q$-Borel-Laplace method of the first kind. 
\begin{rem}We remark that ``the bilateral basic hypergeometric type solution'' of \eqref{bqua2} around the infinity is given by
\[w_\infty (x)={}_1\psi_1\left(0;\frac{q}{a};q,\frac{1}{ax}\right)
=\sum_{n\in\mathbb{Z}}\frac{1}{(q/a;q)_n}\left(\frac{1}{ax}\right)^n.\]
But this solution is not suitable for our argument. We choose and deal with the solution \eqref{b2si}.
\end{rem}
\subsection{An asymptotic formula}
In this section, we show an asymptotic formula of \eqref{b2so} with using the $q$-Borel-Laplace transformations of the first kind. We show the following theorem.
\begin{thm}\label{teiri1}For any $x\in\mathbb{C}^*\setminus [-\lambda ;q]$, we have
\[\hat{{}_1\psi_0}(a;-;\lambda ;q,x)=\frac{(q;q)_\infty}{\left(q/a;q\right)_\infty}\frac{\theta (aq\lambda )}{\theta (q\lambda )}\frac{\theta\left(ax/\lambda \right)}{\theta\left(x/\lambda\right)}\frac{1}{\left(1/ax;q\right)_\infty},\]
where $1<|ax|$.
\end{thm}
\begin{proof}We apply the $q$-Borel transformation $\mathcal{B}_q^+$ to the series ${}_1\psi_0(a;-;q,x)$. Then, 
\[\psi (\xi ):=\left(\mathcal{B}_q^+{}_1\psi_0\right)(\xi )
={}_1\psi_1(a;0;q,-\xi ).\]
By Ramanujan's sum \eqref{Rsum}, we obtain the infinite product representation of $\psi (\xi )$ as follows;
\[\psi (\xi )=\frac{(q;q)_\infty}{(q/a;q)_\infty}\frac{\theta (a\xi )}{\theta (\xi )}(-q/\xi ;q)_\infty .\]
We apply the $q$-Laplace transformation $\mathcal{L}_{q,\lambda}^+$ to $\psi (\xi )$.
\begin{align}\left(\mathcal{L}_{q,\lambda}^+\psi \right)(x)&
=\sum_{n\in\mathbb{Z}}\frac{\psi (\lambda q^n)}{\theta \left(\frac{\lambda q^n}{x}\right)}
=\frac{(q;q)_\infty}{(q/a;q)_\infty}\sum_{n\in\mathbb{Z}}\frac{1}{\theta (\frac{\lambda q^n}{x})}\frac{\theta (a\lambda q^n)}{\theta (\lambda q^n)}\left(-\frac{q}{\lambda q^n};q\right)_\infty\notag \\
&=\frac{(q;q)_\infty}{(q/a;q)_\infty}\frac{1}{\theta (\lambda /x)}\frac{\theta (a\lambda )}{\theta (\lambda )}\sum_{n\in\mathbb{Z}}q^{\frac{n(n-1)}{2}}\left(\frac{\lambda}{ax}\right)^n \left(-\frac{q}{\lambda q^n};q\right)_\infty\notag \\
&=\frac{(q;q)_\infty}{(q/a;q)_\infty}\frac{1}{\theta (\lambda /x)}\frac{\theta (a\lambda )}{\theta (\lambda )}(-q/\lambda ;q)_\infty {}_1\psi_1(-\lambda ;0;q,1/ax).\label{qblpsi}
\end{align}
We remark that 
\[{}_1\psi_1(-\lambda ;0;q,1/ax)=\frac{\theta (\lambda /ax)}{(-q/\lambda ;q)_\infty (1/ax;q)_\infty}.\]
Combining \eqref{qblpsi} and \eqref{inv}, we obtain the conclusion.
\end{proof}

Theorem \ref{teiri1} is rewritten by 
\[\left(\mathcal{L}_{q,\lambda}^+\circ \mathcal{B}_q^+\tilde{u}\right)(x)
=\frac{(q;q)_\infty}{(q/a;q)_\infty}\frac{\theta (aq\lambda )}{\theta (q\lambda )}\frac{\theta (ax/\lambda )}{\theta (x/\lambda )}\frac{\theta (x)}{\theta (ax)}\tilde{u}_\infty (x).\]
We define the function $C(x;q)$ as follows;
\[C(x;q):=\frac{(q;q)_\infty}{(q/a;q)_\infty}\frac{\theta (aq\lambda )}{\theta (q\lambda )}\frac{\theta (ax/\lambda )}{\theta (x/\lambda )}\frac{\theta (x)}{\theta (ax)}.\]
Then, $C(x;q)$ is single valued as a function of $x$.
\begin{cor}The function $C(x;q)$ is the $q$-elliptic function, i.e., $C(x;q)$ satisfies 
\[C(qx;q)=C(x;q),\qquad C(e^{2\pi i}x;q)=C(x;q).\]
\end{cor}
\section{The limit $q\to 1-0$ of the asymptotic formula}
In this section, we give the limit $q\to 1-0$ of our asymptotic formula. We assume that $x\in\mathbb{C}^*\setminus [-\lambda ;q]$ ($-\pi <\arg x<\pi $). We put $a=q^\alpha $ and $x\mapsto x/(1-q)$ in theorem 1 to consider the limit. We remark that the limit of the left-hand side in theorem 1 formally gives the bilateral hypergeometric series ${}_1H_0(\alpha ;-;-x)$. We show the following theorem. \begin{thm}\label{thm2}For any $x\in\mathbb{C}^*\setminus [-\lambda ;q]$ ($-\pi <\arg x<\pi $), we have
\[\lim_{q\to 1-0}
\frac{(q;q)_\infty}{(q^{1-\alpha};q)_\infty}\frac{\theta (q^{\alpha +1} \lambda )}{\theta (q\lambda )}\frac{\theta\left(\dfrac{q^\alpha x}{(1-q)\lambda} \right)}{\theta \left(\dfrac{x}{(1-q)\lambda}\right)}\frac{1}{\left((1-q)\frac{q^{-\alpha}}{x};q\right)_\infty}
=\frac{\Gamma (1-\alpha )}{x^\alpha}e^{\frac{1}{x}},\]
where $1<|x|$.
\end{thm}
We give the proof of theorem \ref{thm2}. 
\begin{proof}
The right-hand side of theorem 1 is rewritten by 
\begin{align}
&\frac{(q;q)_\infty}{(q^{1-\alpha};q)_\infty}
\frac{\theta (q^{\alpha +1} \lambda )}{\theta (q\lambda )}\frac{\theta \left(\frac{q^\alpha}{(1-q)}\frac{x}{\lambda}\right)}{\theta\left(\frac{1}{(1-q)}\frac{x}{\lambda}\right)}\frac{1}{\left((1-q)\frac{q^{-\alpha}}{x};q\right)_\infty}\notag \\
&=\Gamma_q(1-\alpha )\frac{\theta (q^{\alpha +1}\lambda )}{\theta (q\lambda )}
\left\{\frac{\theta \left(\frac{q^\alpha }{(1-q)}\frac{x}{\lambda}\right)}{\theta\left(\frac{q^0}{(1-q)}\frac{x}{\lambda}\right)}
(1-q)^{-\alpha}\right\}
e_q\left((1-q)\frac{1}{q^{\alpha}x}\right).\notag
\end{align}
Combining \eqref{limgamma}, \eqref{limt1}, \eqref{limt2} and \eqref{limexp}, we obtain the conclusion.
\end{proof}

\section*{Acknowledgement} 
My heartfelt appreciation goes to Professor Yousuke Ohyama who provided carefully considered feedback and valuable comments.

\end{document}